\documentclass[12pt,a4paper]{article}
\usepackage[utf8]{inputenc}
\usepackage{tikz}
\usetikzlibrary{arrows,automata}
\tikzset{every state/.style={minimum size=0pt}}
\usepackage{smartdiagram}  
\usepackage{algorithm}
\usepackage{algcompatible}
\usepackage{graphicx}
\usepackage{xcolor}
\usepackage[algo2e,linesnumbered,ruled,vlined]{algorithm2e}
\usepackage[english]{babel}
\usepackage{subfig}
\usepackage{animate}
\usepackage{authblk}
\usepackage[margin=1in]{geometry}
\usepackage[T1]{fontenc}
\usepackage{amssymb, amsthm}
\usepackage{amsmath}
\usepackage{mathtools}
\usepackage{relsize}
\usepackage{todonotes}
\usepackage{hyperref}

\newtheorem{theorem}{Theorem}
\newtheorem{lemma}[theorem]{Lemma}

\newtheorem{observation}[theorem]{Observation}
\newtheorem{proposition}[theorem]{Proposition}

\newtheorem{definition}[theorem]{Definition}
\newtheorem{claim}[theorem]{Claim}

\newtheorem{remark}[theorem]{Remark}

\newcommand{\defproblem}[3]{
  \vspace{3mm}
\noindent\fbox{
  \begin{minipage}{.97\textwidth}
  \begin{tabular*}{\textwidth}{@{\extracolsep{\fill}}lr} \textsc{#1} \\ \end{tabular*}
  {\bf{Input:}} #2  \\
  {\bf{Question:}} #3
  \end{minipage}
  }
  \vspace{2mm}
}

\begin{document}
\title{\textbf{Computational and Combinatorial Results on Conflict-free Choosability 
}}
\author{Shiwali Gupta and Rogers Mathew}
\affil{Department of Computer Science and Engineering, Indian Institute of Technology Hyderabad. 
\authorcr\{cs21resch11002, rogers\}@iith.ac.in}
\date{}
\maketitle

\begin{abstract}
The \emph{conflict-free closed neighborhood (CFCN$^*$) chromatic number} of a graph $G = (V,E)$ is the smallest positive integer $k$ for which there exists a coloring of a subset of vertices using $k$ colors such that, for every vertex in $V$, there exists a color that appears exactly once in its closed neighborhood. The \emph{conflict-free open neighborhood (CFON$^*$) chromatic number} is defined analogously.  In this paper, we study `list variants' of the above-mentioned coloring parameters. The \emph{conflict-free closed neighborhood (CFCN$^*$) choice number} of a graph $G = (V,E)$ is the smallest positive integer $k$ such that for every assignment of lists of size $k$ to its vertices, there exists a coloring of a subset of vertices, say $V'$, in which (i) every vertex in $V'$ receives a color from its list, and (ii) for every vertex in $V$ there exists some color that appears exactly once in its closed neighborhood. 
The \emph{conflict-free open neighborhood (CFON$^*$) choice number} is defined analogously. 

D\k{e}bski and Przyby\l o  [Journal of Graph Theory, 2022] showed that for any graph $G$ with maximum degree $\Delta$, the CFCN$^*$ chromatic number of its line graph is $O(\ln \Delta)$. This result was later extended to claw-free graphs by Bhyravarapu et al. [Journal of Graph Theory, 2025], who proved that every $K_{1,k}$-free graph $G$ admits a CFCN$^*$ coloring using $O(k\ln \Delta)$ colors. 
In this paper, we generalize this result to the list setting and show that every $K_{1,k}$-free graph $G$ has a CFCN$^*$ choice number of $O(k\ln \Delta)$.
Further, we answer some questions concerning the hardness of computing CFCN$^*$/CFON$^*$ choice numbers posed by Gupta and Mathew [SOFSEM, 2026]; in particular, we show that it is NP-hard to determine whether the CFCN$^*$/CFON$^*$ choice number a graph is equal to $k$, for $k=1,2$.
\end{abstract}

\section{Introduction}
We begin by defining the notion of conflict-free (CF) choice number of hypergraphs and the two variants of this parameter on graphs, namely conflict-free open neighborhood (CFON) choice number and conflict-free closed neighborhood (CFCN) choice number. Let $\mathcal{H}=(V,\mathcal{E})$ be a hypergraph.
\begin{definition}[$k$-assignment]
 Let $\mathcal{L}=\{L_v : v\in V\}$, where each $L_v$ is a list of colors.  
We say that $\mathcal{L}$ is a \emph{$k$-assignment} for $\mathcal{H}$ if $|L_v|=k$ for every $v\in V$.
\end{definition}

\begin{definition}[$k$-CF$^*$-choosable]
Let $\mathcal{L}=\{L_v : v\in V\}$ be a list assignment.  
We say that $\mathcal{H}$ admits an \emph{$\mathcal{L}$-CF$^*$-coloring} if there exists a coloring 
$f : V' \rightarrow \bigcup_{v\in V'} L_v$, for some $V'\subseteq V$, such that $f(v)\in L_v$ for all $v\in V'$, and every hyperedge $E\in\mathcal{E}$ contains a vertex whose color is unique within $E$.  
The hypergraph $\mathcal{H}$ is called \emph{$k$-CF$^*$-choosable} if it admits an $\mathcal{L}$-CF$^*$-coloring for every $k$-assignment $\mathcal{L}$.
\end{definition}

\begin{definition}[CF$^*$ choice number]
\label{def:CF_list}
The minimum positive integer $k$ such that $\mathcal{H}$ is $k$-CF$^*$-choosable is called the \emph{CF$^*$-choice number} of $\mathcal{H}$. It is denoted by $ch^*_{CF}(\mathcal{H})$.
\end{definition}

Let $G$ be a finite, simple, undirected graph. Let $V(G)$ denote its vertex set and $E(G)$ denote its edge set. For a vertex $v \in V(G)$, the \emph{open neighborhood} of $v$ in $G$ is the set of all vertices adjacent to $v$, and is denoted by $N_G(v)$. The \emph{closed neighborhood} of $v$ is the set $\{v\} \cup N_G(v)$, and is denoted by $N_G[v]$. 

\begin{definition}[CFON$^*$ choice number]
\label{defn_list_CFON}
Given a list assignment $\mathcal{L} = \{ L_v~:~v \in V(G)\}$ for the graph $G$, we say $G$ is \emph{$\mathcal{L}$-CFON$^*$-colorable} if there exists a coloring  $f: V' \rightarrow \bigcup_{v \in V'}L_v$, where $V' \subseteq V(G)$, such that $f(v) \in L_v$ $\forall v \in V'$,  and every vertex $v$ sees a unique color in its open neighborhood under $f$. We say that $G$ is $k$-CFON$^*$-choosable if for every $k$-assignment $\mathcal{L}$, $G$ is $\mathcal{L}$-CFON$^*$-colorable. The minimum $k$ for which $G$ is $k$-CFON$^*$-choosable is called the \emph{CFON$^*$} \emph{ choice number} of $G$. This is denoted by $ch^*_{ON}(G)$. 
\end{definition}

Analogously, one can define the CFCN$^*$ choice number of a graph $G$ which will be denoted by $ch^*_{CN}(G)$.

Let us define a special list assignment $\mathcal{L}^{[k]}=\{L_v : v\in V\}$, where $L_v = \{1, 2, \ldots , k\}$, for every $v \in V$. That is, under this special list assignment, every vertex is assigned the same list which is $\{1, 2, \ldots , k\}$. We say $\mathcal{H}$ is \emph{$k$-CF$^*$-colorable} if $\mathcal{H}$ admits an $\mathcal{L}^{[k]}$-CF$^*$-coloring. The minimum $k$ for which $\mathcal{H}$ is $k$-CF$^*$-colorable is called the \emph{$CF^*$-chromatic number} of $\mathcal{H}$, denoted by $\chi_{CF}^*(\mathcal{H})$. 
We say $G$ is \emph{$k$-CFON$^*$-colorable} if $G$ is \emph{$\mathcal{L}^{[k]}$-CFON$^*$-colorable}. The minimum $k$ for which $G$ is $k$-CFON$^*$-colorable is called the \emph{CFON$^*$} \emph{chromatic number} of $G$. This is denoted by $\chi_{ON}^*(G)$. Analgously, one can define \emph{CFCN$^*$} \emph{chromatic number} of $G$ which is denoted by $\chi_{CN}^*(G)$. 

\begin{remark}
In the definitions above, we did not insist on coloring all the vertices of the hypergraph (or graph) under consideration. We colored only a subset $V'$ of the set of vertices thus leading to a partial coloring.  The star symbol as a superscript  in the parameters $ch^*_{CF}(\mathcal{H})$, $ch^*_{ON}(G)$, $ch^*_{CN}(G)$, $\chi_{CF}^*(\mathcal{H})$, $\chi^*_{ON}(G)$, and $\chi^*_{CN}(G)$ indicate the underlying partial nature of coloring. That is, we do not insist on coloring all the vertices. However, if we insist on coloring all the vertices, then the corresponding  parameters are denoted $ch_{CF}(\mathcal{H})$, $ch_{ON}(G)$, $ch_{CN}(G)$, $\chi_{CF}(\mathcal{H})$, $\chi_{ON}(G)$, and $\chi_{CN}(G)$, respectively. 
\end{remark}
\begin{remark}
Given a graph $G$, let (i) $\mathcal{H}_{(G)} = (V, \mathcal{E})$, where $V = V(G)$ and $\mathcal{E} = \{N_G(v)~:~v \in V(G)\}$, and (ii) $\mathcal{H}_{[G]} = (V', \mathcal{E}')$, where $V' = V(G)$ and $\mathcal{E}' = \{N_G[v]~:~v \in V(G)\}$. A CFON$^*$ (resp. CFCN$^*$) coloring of $G$ is a CF$^*$ coloring of the hypergraph $\mathcal{H}_{(G)}$ (resp., $\mathcal{H}_{[G]}$). The converse is also true.       
\end{remark}

Let $\mathcal{H}$ (resp. $G$) be a hypergraph (resp., graph) on $n$ vertices. The following propositions connect the various conflict-free parameters with each other.  Proposition \ref{prop:partial vs full} establishes the relationship between partial and full conflict-free coloring. Using this we can extend all the bounds established in this paper to the full coloring setting. 

\begin{proposition}
\textup{\cite{gupta2026bounds}}
\label{prop:partial vs full}
\textnormal{
(i) $\chi_{CF}(\mathcal{H}) \le \chi^*_{CF}(\mathcal{H}) + 1$, 
(ii) $ch_{CF}(\mathcal{H}) = O\!\left(ch_{CF}^*(\mathcal{H}) + \ln n\right)$.}
\end{proposition}
The following proposition connects the conflict-free notions in the classical and list settings.

\begin{proposition}
\textup{\cite{cheilaris2011potential}} 
\label{prop : classical vs full}
$\chi^*_\mathrm{CF}(\mathcal{H}) \le ch^*_\mathrm{CF}(\mathcal{H}) \le  \chi^*_\mathrm{CF}(\mathcal{H}) \cdot \ln n+ 1$.
\end{proposition}
Proposition \ref{prop:CN ON} below relates the CFON$^*$ notion with the CFCN$^*$ notion.

\begin{proposition}
\label{prop:CN ON}
\textnormal{ 
    (i) \cite{pach2009conflict} $\chi^*_{CN}(G) \le 2\,\chi^*_{ON}(G)$,
    (ii) \cite{gupta2026bounds} $ch^*_{CN}(G) \le 2\,ch^*_{ON}(G)\cdot \ln n + 1$.}
\end{proposition}

Motivated by its applications in the frequency-assignment problem in wireless networks, conflict-free coloring was first studied by Even et al. \cite{even2003conflict} in 2003. Conflict-free coloring was later studied in several different settings (see \cite{cheilaris2009conflict, glebov2014conflict, pach2009conflict, bhyravarapu2021short, dkebski2022conflict, bhyravarapu2022conflict, LEVTOV20091521, ElbassioniM06, pach2003conflict, har2003conflict, alon2006conflict}).
For a survey on conflict-free coloring and its applications, see Smorodinsky \cite{smorosurvey}. Conflict-free coloring has found applications in various areas, including frequency assignment in cellular networks, energy-efficient sensor networks, vertex ranking with applications in VLSI design and operations research, and the pliable index coding problem in coding theory \cite{krishnan2022pliable}.
The classical conflict-free coloring framework assumes that all vertices can choose colors from a common global set, an assumption that is often too restrictive in practical settings. This motivates the study of the list variant, in which each vertex is assigned a list of permissible colors, providing a more realistic and general model of constrained coloring scenarios. For detailed applications of conflict-free choosability, see \cite{gupta2026bounds}.

\subsection{Background and Our Contribution}
Conflict-free choosability was first studied by Cheilaris, Smorodinsky, and Sulovsk{\`y} \cite{cheilaris2011potential} in the context of certain geometric hypergraphs. Gupta and Mathew \cite{gupta2026bounds} showed that, for a graph $G$ with maximum degree $\Delta$, the CFCN$^*$ choice number of $G$ is $O(\ln^2 \Delta)$ thus generalizing a similar result of Bhyravarapu et al. \cite{bhyravarapu2021short} for the CFCN$^*$ chromatic number of $G$.  
Recently, Wu and Zhang \cite{wu2025list} showed that for every planar graph $G$, $ch_{ON}(G) \le 12$. 

D\k{e}bski and Przyby\l o \cite{dkebski2022conflict} showed that, for any graph $G$ with maximum degree $\Delta$, $\chi^*_{CN}(L(G)) = O(\ln \Delta)$ where $L(G)$ denotes the line graph of $G$. Line graphs are known to be a subclass of claw-free (or $K_{1,3}$-free) graphs. It was asked as an open question in \cite{dkebski2022conflict} to extend the above result to claw-free graphs. Bhyravarapu et al. \cite{bhyravarapu2025extremal} solved this open question by showing that if $G$ is $K_{1, k}$-free with maximum degree $\Delta$, then $\chi^*_{CN}(G) = O(k\ln \Delta)$. In this paper, we generalize this result to a list coloring setting. We prove that if $G$ is a $K_{1,k}$-free graph with maximum degree $\Delta$, then $ch^*_{CN}(G)=O(k\ln \Delta)$.

In Section \ref{Hardness Results}, we investigate the hardness of decision problems related to CFON/CFCN choosability and their partial variants which are defined below. 

\defproblem{\textsc{$k$-CFON$^*$-choosability problem ($k$-CFON$^*$-CH PROBLEM)}}{A graph $G$ and a positive integer $k$.}{Is $ch^*_{ON}(G) \le k$?}

\defproblem{\textsc{$k$-CFON-choosability problem ($k$-CFON-CH PROBLEM)}}{A graph $G$ and a positive integer $k$.}{Is $ch_{ON}(G) \le k$?}

Analogously, we define the \textsc{$k$-CFCN$^*$-choosability problem ($k$-CFCN$^*$-CH PROBLEM)} and the \textsc{$k$-CFCN-choosability problem ($k$-CFCN-CH PROBLEM)}.

Gupta and Mathew \cite{gupta2026bounds} proved that the following decision problems are $\Pi_2^{P}$-complete:
(i) the \textsc{$k$-CFON-choosability} problem,
(ii) the \textsc{$k$-CFON$^*$-choosability} problem, and
(iii) the \textsc{$(2,k)$-CFCN-choosability} problem.
These hardness results hold for bipartite graphs for any fixed $k \geq 3$, for planar triangle-free graphs when $k=3$, and for planar graphs when $k=4$. They also posed the following problems as open questions. The computational complexity of
(i) \textsc{$k$-CFON$^*$-choosability} for $k=1,2$,
(ii) \textsc{$k$-CFON-choosability} for $k=2$,
(iii) \textsc{$k$-CFCN$^*$-choosability} for all $k \ge 1$, and
(iv) \textsc{$k$-CFCN-choosability} for all $k \ge 2$
remains unresolved.

In this paper, we make progress in some of the open questions posed by Gupta and Mathew \cite{gupta2026bounds}. In particular, we show that the following are NP-hard: 
(i) \textsc{$k$-CFON$^*$-choosability} for $k=1,2$, and
(ii) \textsc{$k$-CFCN$^*$-choosability} for $k=1,2$.

\subsection{Preliminaries}
\subsubsection{Definitions and Notations}
We study simple, finite, and undirected graphs throughout this paper. When discussing open neighborhood coloring, we assume the graph under consideration has no isolated vertices. The \emph{degree} of a vertex $v \in V$ in a given hypergraph $\mathcal{H} = (V, \mathcal{E})$, denoted by $d_{\mathcal{H}}(v)$, is the number of hyperedges that $v$ is present in. The \emph{maximum degree} of $\mathcal{H}$ is defined as $\max \{d_\mathcal{H}(v) :v \in V\}$. Let $K_{1, k}$ denote the complete bipartite graph with one vertex in one part and $k$ vertices in other part. A graph is called claw-free if it does not contain an induced subgraph isomorphic to the claw graph $K_{1,3}$. 

\subsubsection{Probabilistic Inequalities and Known Results}
The following results on conflict-free coloring are used in the proof of Theorem \ref{claw free graph}.  

\begin{theorem} \textup{\cite{cheilaris2011potential}}
\label{max degree + 1}
For any hypergraph $\mathcal{H}$ with maximum degree at most $\Delta$, $ch_{CF}(\mathcal{H}) \leq \Delta + 1$. 
\end{theorem}

\begin{theorem} \textup{\cite{gupta2026bounds}}
\label{thm: general cfcn upper bound}
For any graph $G$ with maximum degree $\Delta$, $ch^*_\mathrm{CN}(G) = O(\ln^2 \Delta)$.
\end{theorem}

We state the Local Lemma \cite{lovaszlocallemma} and a variant of the Talagrand’s inequality due to \cite{MolloyR14} below. These are  used in the proof of Lemma~\ref{lem_near_uniform_hypergraph}. 

\begin{lemma}[\emph{The Local Lemma}, \cite{lovaszlocallemma}] \label{lem:local} Let $A_1, \ldots , A_n$ be events in an arbitrary probability space. Suppose that each event $A_i$ is mutually independent of a set of all the other events $A_j$ but at most $d$, and that $Pr[A_i] \leq p$ for all $i \in [n]$. If  $4pd \leq 1$, then $Pr[\cap _{i=1}^n \overline{A_i}] > 0$.  
\end{lemma}

\begin{theorem}
[Talagrand's Inequality, \cite{MolloyR14}]
\label{thm_Talagrand}
Let $X$ be a non-negative random variable determined by the independent trials $T_1, \ldots, T_n$. Suppose that for every set of possible outcomes of the trials, we have:
\\
\textnormal{(i) changing the outcome of any one trial can affect $X$ by at most $a$; and
\\(ii) for each $s>0$, if $X \geq s$ then there is a set of at most $bs$ trials whose outcomes certify that $X \geq s$.}
\\Then for any $t \geq 0$, we have
$$Pr[|X - E[X]| > t + 20a\sqrt{b E[X]} + 64a^2b] \leq 4e^{-\frac{t^2}{8a^2b(E[X] + t)}}.$$
\end{theorem}

\section{Combinatorial Results for Conflict-free Choosability}

The following lemma gives a coloring of the vertices of a `nearly uniform' hypergraph in which every hyperedge $E$ sees at least $\frac{|E|}{8}$ distinctly colored vertices.
This lemma plays a crucial role in proving Theorem~\ref{claw free graph}. 
 
\begin{lemma} 
\label{lem_near_uniform_hypergraph}
Let $\mathcal{H} = (V,\mathcal{E})$ be a hypergraph satisfying the below conditions: 
\\
(i) Every hyperedge intersects with at most $\Gamma$ other hyperedges, and \\
(ii) For every hyperedge $E \in \mathcal{E}$, $\alpha \le |E| \leq \beta$, where $\alpha = \max (2^{12}, \lceil 136 \ln(16 \Gamma)\rceil)$.\\
Let $\mathcal{L}=\{L_v~:~v \in V(\mathcal{H})\}$ be a $32 \beta$-assignment for $\mathcal{H}$. Then there is an $\mathcal{L}$-coloring of $\mathcal{H}$ in which every hyperedge $E$ sees at least $\frac{|E|}{8}$ unique colors. 
\end{lemma}

\begin{proof}
For each vertex in $V$, assign a color that is chosen independently, uniformly at random from its list of size $32\beta$. For any hyperedge $E \in \mathcal{E}$, let $X_E$ be a random variable that denotes the number of vertices in $E$ whose color is not unique in $E$. Thus, for any vertex $v$ in $E$, we have
\begin{align*}
Pr[\text{color of $v$ is not unique in $E$}] & \le 1 - \Bigg(1 - \frac{1}{32\beta}\Bigg)^{|E|-1} \leq 1 - \Bigg(1 - \frac{|E|-1}{32\beta}\Bigg) \leq \frac{|E|}{32\beta}. 
\end{align*}
Then, by linearity of expectation,
\begin{align*}
E[X_E] & \le \frac{|E|^2}{32\beta} \leq \frac{|E|}{32}. 
\end{align*}
We claim that the random variable $X_E$ satisfies the assumptions of Theorem \ref{thm_Talagrand} (Talagrand's Inequality) with $n=|E|$, $a=2$, and $b=2$. The value of $X_E$ is determined by $|E|$ independent trials. Changing the outcome of any trial can affect $X_E$ by at most $2$. For any $s>0$, if it is given that $X_E \geq s$, then there is a set of at most $2s$ trials whose outcomes would ensure that $X_E$ is at least $s$, regardless of the outcomes of the remaining trials. This proves our claim. Let $A = \frac{|E|}{2} + 20a\sqrt{bE[X]}  + 64a^2b = \frac{|E|}{2} + 40\sqrt{2E[X]}  + 512$. Applying Theorem \ref{thm_Talagrand} with $t = |E|/2$, we get 
\begin{eqnarray*}
Pr[|X_E - E[X_E]| > A ] & \leq & 4e^{-\frac{|E|^2/4}{64(E[X_E] + |E|/2)}} \\
 & \leq & 4e^{-\frac{|E|^2}{256(17|E|/32)}}~~(\mbox{since }E[X_E] \leq \frac{|E|}{32}) \\
 & \leq & 4e^{-\frac{|E|}{136}} \\
 & \le & 4e^{-\ln (16 \Gamma)}~~(\mbox{since } |E| \geq \alpha \geq \lceil 136 \ln(16 \Gamma)\rceil) \\  
 & = & \frac{1}{4\Gamma}.
\end{eqnarray*}
Since $E[X_E] \leq \frac{|E|}{32}$, we have 
$$E[X_E] + A \leq \frac{17|E|}{32} + 40\sqrt{\frac{2|E|}{32}} + 512 = \frac{17|E|}{32} + 10\sqrt{|E|} + 512 < \frac{7|E|}{8},$$ 
where the last inequality follows from the fact that $|E| \geq \alpha \geq 2^{12}$. Thus, 
\begin{eqnarray*}
Pr\Bigg[X_E \ge \frac{7|E|}{8}\Bigg] & = & Pr\Bigg[X_E - \frac{7|E|}{8} \ge 0\Bigg] \\
 & \leq & Pr[X_E - ({E[X_E] + A}) > 0] \\
 & = & Pr[(X_E - E[X_E]) > A] \\
 & \leq &  Pr[|X_E - E[X_E]| > A] \\  
 & \leq & \frac{1}{4\Gamma}.
\end{eqnarray*}
Let $A_E$ denote the bad event that $X_E \ge \frac{7|E|}{8}$. From the above calculations, we know that $Pr[A_E] \leq \frac{1}{4\Gamma}$. 
We can apply the Local Lemma (Lemma \ref{lem:local}) on the events $A_E$, for all hyperedges $E \in \mathcal E$.
Since each hyperedge intersects with at most $\Gamma$ other hyperedges, and $4\cdot \frac{1}{4\Gamma}\cdot \Gamma \leq 1$,  
we get $Pr[\cap_{E \in \mathcal{E}}(\overline A_E)] > 0$. 
Thus, $\mathcal{H}$ admits an $\mathcal{L}$-coloring for any $32\beta$-assignment $\mathcal{L}$ in which every hyperedge $E$ in $\mathcal{H}$ sees at least $\frac{|E|}{8}$ unique colors.
\end{proof}

We show an upper bound for the CFCN$^*$ choice number for
$K_{1,k}$-free graphs.

\begin{theorem} 
\label{claw free graph} 
Let $G$ be a $K_{1,k}$-free graph with maximum degree $\Delta$. Then $ch^*_{CN}(G) = O(k\ln \Delta)$. 
\end{theorem}

\begin{proof}
Suppose $k > \frac{\ln \Delta}{8}$.  Then by Theorem \ref{thm: general cfcn upper bound}, $ch^*_{CN}(G) = O(k \ln \Delta)$. For the rest of the proof, we shall assume that $k \le \frac{\ln \Delta}{8}$.
Construct a maximal independent set $A$ of $G$. We remove all the vertices that belong to $A$ from $G$ to obtain a new graph $G'$.

\begin{observation}
\label{obs: nbr in A}
    Since $A$ is a maximal independent set in $G$, every vertex in $G'$ has some neighbor in $A$. Further, since $G$ is $K_{1, k}$-free, every vertex in $G'$ has at most $k-1$ neighbors in $A$.
\end{observation}

We consider a greedy, proper coloring of $G'$ using $s$ colors, where $1 \leq s \leq \Delta +1$, as described below. 
Vertices are colored one by one, and each vertex is given the smallest available color that is not used by any of its colored neighbors so far. Let $S_1, S_2, \dots, S_s$ be the color classes given by this coloring.

\begin{observation}
\label{obs: nbr in color class s}
Let $2 \leq i \leq s$. Every vertex in the color class $S_i$ has at least one neighbor in each color class $S_j$, for every $j < i$.  \end{observation}

We prove the observation by contradiction. Suppose $v \in S_i$ had no neighbor in $S_j$, for some $j < i$. Then the greedy algorithm would have given $v$ a color of value at most $j$.

\begin{observation}
\label{obs: k-1 nbrs in each s_i}
Since $G$ is $K_{1,k}$-free, every vertex in $G$ has at most $k - 1$ neighbors in each $S_i$, where $1 \le i \le s$. 
\end{observation}

We partition the color classes obtained from the above proper coloring of $G'$ into two parts, $B$ and $C$. Let $B := S_1 \cup S_2 \cup \dots \cup S_{b}$, where $b = \min \{s,  \max\{2^{12}, 272 \ln (4\Delta)\}\}$. 
Thus, $b\le s$. 
Let $C := S_{b+1} \cup \cdots \cup S_s$. Suppose $b=s$. Then, $C = \emptyset$ which makes the proof easier as we won't need to find uniquely colored neighbors for the vertices in $C$. Further, when $C= \emptyset$, our proof will go through as we do not color any vertex in $C$ in our proof and therefore no vertex will have its uniquely colored neighbor in $C$. For the rest of the proof, we assume that $b<s$ and therefore (i) $b = \max\{2^{12}, 272 \ln (4\Delta)\}$, and (ii) $C \neq \emptyset$.  
We make the following observation which follows from Observations \ref{obs: nbr in color class s} and \ref{obs: k-1 nbrs in each s_i}.

\begin{observation}
 \label{obs: nbr in B}  
 Every vertex in $C$ has at least $b$ neighbors in $B$. Moreover, for every vertex $v\in V(G)$, the number of neighbors of $v$ that belongs to $B$ is at most $(k-1)b$.
\end{observation}

Let $\mathcal{L} = \{L_v~:~v \in V(G) \}$ be an $r$-assignment for $G$ given by the adversary, where $r = 2^{18} k \ln \Delta$.  We obtain the desired $\mathcal{L}$-CFCN$^*$-coloring of $G$ by list CF$^*$ coloring two hypergraphs $\mathcal{H}_1$ and $\mathcal{H}_2$ that are defined below.

\begin{itemize}
    \item Let $\mathcal{H}_1 = (V_1, \mathcal{E}_1)$, where $V_1 = A$ and $\mathcal{E}_1 = \{N_{G}[v] \cap A : v \in (A \cup B)\}$. By Observation \ref{obs: nbr in B}, the maximum degree of $\mathcal{H}_1$ is at most $(k-1)b +1$. Thus, by Theorem \ref{max degree + 1}, we have $ch^*_{CF}(\mathcal{H}_1) \le (k-1)b + 2$. Thus $\mathcal{H}_1$ is $\mathcal{L}$-CF$^*$-colorable. 
   Let $f_1$ denote this coloring. Under $f_1$, every vertex in $A \cup B$ sees a unique color in its closed neighborhood. 
\end{itemize}

Next, by list CF$^*$ coloring a hypergraph $\mathcal{H}_2$ (defined below) having vertex set $B$, we intend to take care of every vertex in $C$. This however can lead to two problems that are described below: 
\\ 
(i) For a vertex $a \in A$, let $c_a$ denote the unique color seen by $a$ under the coloring $f_1$. We should ensure that none of the neighbors of $a$ in $B$ receive the color $c_a$. For this, we update the lists of every vertex $u \in B$ as $L'_u = L_u \setminus X_u$, where $X_u = \{c_a : a \in N_G(u) \cap A\}$. By Observation \ref{obs: nbr in A}, we have $|X_u| \le (k-1)$.
\\
(ii) For a vertex $w \in B$, let $c_w$ denote the unique color seen by $w$ under the coloring $f_1$. We should ensure that none of the closed neighbors of $w$ in $B$ receive the color $c_w$. For this, we update the lists of every vertex $u \in B$ as $L''_u = L'_u \setminus Y_u$, where $Y_u = \{c_w : w \in N_G[u] \cap B\}$. By Observation \ref{obs: nbr in B}, we have  $|Y_u| \le (k-1)(b-1) +1$. 
Thus, for every vertex $u \in B$, we remove at most $|X_u|+ |Y_u| \le (k-1) + (k-1)(b-1) +1 = (k-1)b +1$ colors from $u's$ list in total.

\begin{itemize}
    \item Thus $\mathcal{L}''= \{L_v'' ~:~ v \in V(B)\}$ is a $r'$-assignment for every vertex $v$ in $B$, where $r' = r-((k-1)b +1) \ge 2^{17} k \ln \Delta$. 
    Let $\mathcal{H}_2 = (V_2, \mathcal{E}_2)$, where $V_2 = B$ and $\mathcal{E}_2 = \{N_{G'}(v) \cap B : v \in C\}$. 
    Fix a vertex $v \in C$. Notice that $v$ shares a common neighbor with at most $\Delta^2$ other vertices in $C$. Thus, every hyperedge in $\mathcal{H}_2$ overlaps with at most $\Gamma$ other hyperedges, where $\Gamma \leq \Delta^2$. 
    
    By Observation \ref{obs: nbr in B}, any vertex $v \in C$ has at least $b$ and at most $(k-1)b$ neighbors in $B$. So for each $E \in \mathcal E_2$, we have $b \leq |E| \leq (k-1)b$. Further, $32 (k-1)b = \max \{2^{17}(k-1), 8704 (k-1)\ln (4\Delta)\} < 2^{17}k \ln \Delta \le r'$.  By applying Lemma \ref{lem_near_uniform_hypergraph} to $\mathcal H_2$, with $\alpha = b$, $\beta = (k-1)b$, and $\Gamma \leq \Delta^{2}$, 
    we get an $\mathcal{L}''$-coloring of $\mathcal{H}_2$ such that every hyperedge $E$ sees at least $\frac{|E|}{8}$ unique color.
   From the definition of $\mathcal{H}_2$, it follows that every vertex $v$ in $C$ is getting to see at least $\frac{b}{8}$ uniquely colored vertices among its neighbors in $B$.
    Since $b \ge 272  \ln(4 \Delta) \ge 2176k$ (since $k \le \frac{\ln \Delta}{8}$),  $v$ sees at least $272k$ uniquely colored vertices among its neighbors in $B$. Some of these colors may be used by a neighbor of $v$ in $A$. Since $v$ has at most $k-1$ neighbors in $A$, $v$ still sees at least $272k - (k-1)$ uniquely colored vertices in its neighborhood.
\end{itemize}

The above coloring ensures that every vertex sees a uniquely colored vertex in its closed neighborhood. In addition, each vertex in the above coloring process receives a color at most once. Some vertices, that are not part of any hypergarph, remain uncolored. 
\end{proof}

\section{Computational Results for Conflict-free Choosability}
\label{Hardness Results}
This section studies the computational complexity of the CFON$^*$/CFCN$^*$ choosability problem.

Mulzer and Rote \cite{mulzer2008minimum} showed that  \textsc{Positive Planar 1-in-3-SAT} is NP-hard. 
We define a few necessary notions before introducing \textsc{Positive Planar 1-in-3-SAT}. Given a 3-CNF Boolean formula $\phi$ on $n$ variables, say $X = \{x_1, \ldots , x_n\}$,  and $m$ clauses, say $C = \{c_1, \ldots , c_m\}$, its \emph{associated graph} $G_\phi$ is a bipartite graph on $n+m$ vertices with bipartition $\{X,C\}$ where a vertex $x_i$ is connected by an edge to a vertex $c_j$ if and only if $x_i$ or $\overline{x_i}$ is present in the clause $c_j$. Throughout this section, when we refer to $x_i$ (or $c_i$), it can either be the variable $x_i$ (resp., the clause $c_i$) in $\phi$ or the corresponding vertex $x_i$ (resp., $c_i$) in $G_\phi$. This will either be explicitely mentioned, or it will be clear from the context.   
 A Boolean formula $\phi$ is said to be \emph{planar} if its associated graph $G_\phi$ is planar. A Boolean formula $\phi$ is called \emph{positive} if every literal in $\phi$ is positive.  

\defproblem{\textsc{Positive Planar 1-in-3-SAT Problem}[see \cite{mulzer2008minimum}]}{A positive planar $3$-CNF Boolean formula $\phi$ on variables $X = \{x_1, \ldots , x_n\}$ and clauses $C = \{c_1, \ldots , c_m\}$ of the form $$\phi = c_1 \land c_2 \land \dots \land c_m,$$
where each $c_i$ is of the form $(x_{i} \lor x_{j} \lor x_{k})$.  
}{Does there exist a truth assignment to the variables such that each clause sees exactly one \textsc{True} variable?}

The only connected  graph with a CFON-choice-number equal to 1 is a path on 2 vertices. Therefore, deciding whether a
graph is 1-CFON choosable can be done in polynomial time. Similarly, the only graph with a CFCN-choice-number equal to 1 is the edgeless graph, i.e., a graph consisting solely of isolated vertices. Therefore, deciding whether a graph is 1-CFCN choosable can be done in polynomial time. 

Abel et al.~\cite{abel2018conflict} proved that it is NP-complete to decide whether  
(i) a planar bipartite graph is $1$-CFON$^{*}$-colorable, and  
(ii) a planar bipartite graph is $1$-CFCN$^{*}$-colorable. Below we prove a proposition that equates the notion of $1$-CFON$^{*}$ (resp. $1$-CFCN$^{*}$) colorability with $1$-CFON$^{*}$ (resp. $1$-CFCN$^{*}$) choosability.
  
\begin{proposition}
\label{prop: 1 coloring choosability same}
\textnormal{
 (i) A graph $G$ is $1$-CFON$^{*}$-colorable if and only if $G$ is $1$-CFON$^{*}$-choosable.\\
(ii) A graph $G$ is $1$-CFCN$^{*}$-colorable if and only if $G$ is $1$-CFCN$^{*}$-choosable.}
\end{proposition}

\begin{proof}
We prove~(i); the proof of~(ii) is analogous.
First, suppose that $G$ is $1$-CFON$^{*}$-colorable. We show that $G$ is $1$-CFON$^{*}$-choosable. Let $S \subseteq V(G)$ be the set of vertices that receive a color in a $1$-CFON$^{*}$-coloring of $G$. Let $\mathcal{L} = \{L_v : v \in V(G)\}$ be a $1$-assignment for $G$. We color each vertex in $S$ with the only color in its list. This produces a valid $\mathcal{L}$-CFON$^{*}$-coloring of $G$.

Conversely, suppose that $G$ is $1$-CFON$^{*}$-choosable. We show that $G$ is $1$-CFON$^{*}$-colorable. Let $\mathcal{L} = \{L_v : v \in V(G)\}$ be a $1$-assignment where $L_v = \{R\}$ for every vertex $v \in V(G)$. Let $Q \subseteq V(G)$ be the set of vertices that receive color $R$ under an $\mathcal{L}$-CFON$^*$-coloring of $G$. Note that the coloring obtained is indeed a valid $1$-CFON$^{*}$-coloring of $G$.
\end{proof}

Thus, by Proposition~\ref{prop: 1 coloring choosability same} and Theorems 1.3 and 6.5 in Abel et. al. \cite{abel2018conflict}, we conclude that  
(i) the \textsc{$1$-CFON$^{*}$-choosability problem} on planar bipartite graphs is NP-complete, and  
(ii) the \textsc{$1$-CFCN$^{*}$-choosability problem} on planar bipartite graphs is NP-complete.
We provide alternate proofs to these problems.
We present a significantly simpler construction than the one used by Abel et al.~\cite{abel2018conflict}.

\begin{lemma}
\label{lem:1CFON*-NP}
\textsc{1-CFON$^*$-CH PROBLEM} is in NP. 
\end{lemma} 
\begin{proof}
Given a graph $G$, a subset $S$ of vertices of $G$ is a \emph{Perfect Induced Matching Dominating Set} (PIMDS) for $G$ if (i) the subgraph of $G$ induced by $S$ is a matching, and (ii) every vertex in $G$ has exactly one neighbor in $S$. It can be verified that $G$ is  1-CFON$^*$-choosable if and only if $G$ contains a PIMDS, a property which can be verified in polynomial time.       
\end{proof} 

\begin{theorem}
\label{thm 1 cfon* chhosability}
\textsc{1-CFON$^*$-choosability problem} is NP-hard on planar bipartite graphs.    
\end{theorem}

\begin{proof}
We give a polynomial-time reduction from the \textsc{Positive Planar 1-in-3-SAT Problem}. 

\noindent\textbf{Construction:} From an input 3-CNF Boolean formula $\phi$ on a variable set $X = \{x_1, \ldots, x_n\}$ and a clause set $C = \{c_1, \ldots, c_m\}$ for the \textsc{Positive Planar 1-in-3-SAT Problem}, we construct its associated graph $G_\phi$ which is known to be planar and bipartite. We modify $G_\phi$ as described below. For each $i \in \{1, \ldots, n\}$, we attach a distinct path on $2$ vertices to $x_i$ (thus forming an induced path on $3$ vertices called the \emph{variable gadget of $x_i$}). The resultant planar bipartite graph is called $G'_\phi$. It has $m+3n$ vertices. 
Figure \ref{fig 1 on*} shows the graph $G'_{\phi}$, where $\phi = (x_1 \lor x_2 \lor x_3) \land (x_1 \lor x_2 \lor x_5) \land (x_1 \lor x_3 \lor x_5) \land (x_3 \lor x_4 \lor x_5)$.

\begin{figure}[!h]
\centering
\begin{tikzpicture}
 \node[shape=circle,draw=black, thin, minimum size = 0.2cm, inner sep=0pt] (1) at (0,0) {};
 \node[shape=circle,draw=black, fill=black, thin, minimum size = 0.2cm, inner sep=0pt] (2) at (1,0) {};
 \node[shape=circle,draw=black, fill=black, thin, minimum size = 0.2cm, inner sep=0pt] (3) at (2,0) {};

  \node[shape=circle,draw=black, fill=black, thin, minimum size = 0.2cm, inner sep=0pt] (1a) at (3,0) {};
 \node[shape=circle,draw=black, fill=black, thin, minimum size = 0.2cm, inner sep=0pt] (2a) at (4,0) {};
 \node[shape=circle,draw=black, thin, minimum size = 0.2cm, inner sep=0pt] (3a) at (5,0) {};

  \node[shape=circle,draw=black, fill=black, thin, minimum size = 0.2cm, inner sep=0pt] (1b) at (6,0) {};
 \node[shape=circle,draw=black, fill=black, thin, minimum size = 0.2cm, inner sep=0pt] (2b) at (7,0) {};
 \node[shape=circle,draw=black, thin, minimum size = 0.2cm, inner sep=0pt] (3b) at (8,0) {};

  \node[shape=circle,draw=black, thin, minimum size = 0.2cm, inner sep=0pt] (1c) at (9,0) {};
 \node[shape=circle,draw=black, fill=black, thin, minimum size = 0.2cm, inner sep=0pt] (2c) at (10,0) {};
 \node[shape=circle,draw=black, fill=black, thin, minimum size = 0.2cm, inner sep=0pt] (3c) at (11,0) {};

  \node[shape=circle,draw=black, fill=black, thin, minimum size = 0.2cm, inner sep=0pt] (1d) at (12,0) {};
 \node[shape=circle,draw=black, fill=black, thin, minimum size = 0.2cm, inner sep=0pt] (2d) at (13,0) {};
 \node[shape=circle,draw=black, thin, minimum size = 0.2cm, inner sep=0pt] (3d) at (14,0) {};

   \path [-](1) edge node[left] {} (2);
   \path [-](2) edge node[left] {} (3); 
   \path [-](1a) edge node[left] {} (2a);
   \path [-](2a) edge node[left] {} (3a); 
   \path [-](1b) edge node[left] {} (2b);
   \path [-](2b) edge node[left] {} (3b);
   \path [-](1c) edge node[left] {} (2c);
   \path [-](2c) edge node[left] {} (3c);
   \path [-](1d) edge node[left] {} (2d);
   \path [-](2d) edge node[left] {} (3d);
   
   \node [black, below] at (2, -0.1) {$x_1$};
   \node [black, below] at (5, -0.1) {$x_2$};
   \node [black, below] at (8, -0.1) {$x_3$};
   \node [black, below] at (11, -0.1) {$x_4$};
   \node [black, below] at (14, -0.1) {$x_5$};

  \node[shape=circle,draw=black, thin, minimum size = 0.2cm, inner sep=0pt] (c1) at (3,1) {};
  \node[shape=circle,draw=black, thin, minimum size = 0.2cm, inner sep=0pt] (c2) at (6.5,-2) {};
  \node[shape=circle,draw=black, thin, minimum size = 0.2cm, inner sep=0pt] (c3) at (10,2) {};
  \node[shape=circle,draw=black, thin, minimum size = 0.2cm, inner sep=0pt] (c4) at (10,-1) {};

  \node [black, above] at (3, 1.1) {$c_1$};
  \node [black, below] at (6.5, -2.1) {$c_2$};
  \node [black, above] at (10, 2.1) {$c_3$};
  \node [black, below] at (10, -1.1) {$c_4$};

   \draw [-](c1) to [] (3);
   \draw [-](c1) to [] (3a);
   \draw [-](c1)  to [] (3b);
   \draw [-](c2)  to [] (3);
   \draw [-](c2)  to []  (3a);
   \draw [-](c2)  to [out=0, in=320]  (3d);
   \draw [-](c3)  to [out=180, in=120] (3);
   \draw [-](c3)  to [] (3d);
   \draw [-](c3)  to [] (3b);
   \draw [-](c4)  to []  (3b);
   \draw [-](c4) to [] (3c);
   \draw [-](c4)  to []  (3d);
   
\end{tikzpicture}
\caption{The graph $G'_{\phi}$, where $\phi= $  
$(x_1 \lor x_2 \lor x_3) \land (x_1 \lor x_2 \lor x_5) \land (x_1 \lor x_3 \lor x_5) \land (x_3 \lor x_4 \lor x_5)$. $\phi$ is satisfiable by setting $x_1, x_4$ to be true. The black vertices will receive a color from their respective lists in any valid CFON$^*$-coloring of the graph.}
\label{fig 1 on*}
\end{figure}
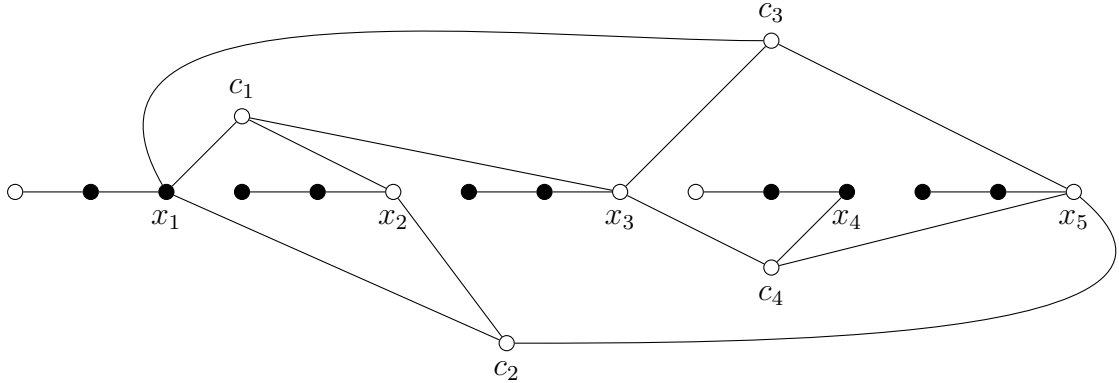

\begin{claim}
\label{claim 1 cfon* chhosability}
The formula $\phi$ is a yes instance of the \textsc{Positive Planar 1-in-3-SAT Problem} if and only if $G'_{\phi}$ is 1-CFON$^*$-choosable.    
\end{claim}
\noindent \textbf{Proof of claim.}
Suppose $\phi$ is a yes instance, that is, there is an assignment of truth values to its variables where each clause sees exactly one \textsc{True} variable and two \textsc{False} variables. Let $\mathcal{L}= \{L_v ~:~ v \in V(G'_{\phi})\}$ be a $1$-assignment for $G'_{\phi}$. Let $i \in \{1, \ldots n\}$. If $x_i=\textsc{True}$, then $\mathcal{L}$-color the vertex $x_i$ and its only neighbor in the variable gadget of $x_i$. Otherwise, leave $x_i$ uncolored and $\mathcal{L}$-color the other two vertices in the variable gadget of $x_i$. The reader may verify that, under this coloring, every vertex has exactly one vertex in its open neighborhood that is colored.  

Suppose, $G'_{\phi}$ is 1-CFON$^*$-choosable. Let $\mathcal{L}= \{L_v ~:~ v \in V(G'_{\phi})\}$ be a $1$-assignment for $G'_{\phi}$, where $L_v = \{c\}$, for every $v \in  V(G'_{\phi})$. Consider an $\mathcal{L}$-CFON$^*$ coloring $f$ of $G'_{\phi}$. We observe that, under the coloring $f$, the variable gadget for each variable $x_i$ has exactly two vertices colored; either the middle vertex and the vertex to its right which is $x_i$ or the middle vertex and the vertex to its left. We set $x_i = \textsc{True}$ if the vertex $x_i$ is colored under $f$. Otherwise, we set $x_i = \textsc{False}$. Since $f$ is a valid CFON$^*$ coloring of $G'_{\phi}$, every clause vertex has exactly one of its neighbors colored. Thus, under the above truth  assignment, every clause will see exactly one \textsc{True} variable.   
\end{proof}

Below, we show that the \textsc{1-CFCN$^*$-CH PROBLEM} is NP-hard for planar bipartite graphs (by reducing from the \textsc{Positive Planar 1-in-3-SAT Problem}). 

\begin{lemma}
\label{lem:1CFCN*-NP}
\textsc{$1$-CFCN$^*$-CH PROBLEM} is in NP. 
\end{lemma} 
\begin{proof}
Given a graph $G$, a \emph{Perfect Independent Dominating Eet (PIDS)} is a set of vertices $S \subseteq V(G)$ such that $S$ is an independent set and every vertex outside $S$ has exactly one neighbor in $S$. That is, for each $v \in V(G)$, $|N_G(v) \cap S| =1$. 
It can be verified that $G$ is  1-CFCN$^*$-choosable if and only if $G$ contains a PIDS, a property which can be verified in polynomial time.  
\end{proof}

\begin{theorem}
\label{thm 1 cfcn* chhosability planar bipar}
    \textsc{$1$-CFCN$^*$-choosability problem} is NP-hard on planar bipartite graphs. 
\end{theorem}

\begin{proof}
We give a polynomial-time reduction from the \textsc{Positive Planar 1-in-3-SAT Problem}.

\noindent \textbf{Construction:} From an input 3-CNF Boolean formula $\phi$ on a variable set $X = \{x_1, \ldots, x_n\}$ and a clause set $C = \{c_1, \ldots, c_m\}$ for the \textsc{Positive Planar 1-3-SAT Problem}, we construct its associated graph $G_\phi$ which is known to be bipartite. We modify $G_\phi$ as described below. For each $i \in \{1, \ldots, n\}$, we attach a pendant vertex $v_i$ to $x_i$. The resultant bipartite graph is called $G''_\phi$. It has $m+2n$ vertices. Figure \ref{fig 1 cn*} shows the graph $G''_{\phi}$, where $\phi = (x_1 \lor x_2 \lor x_3) \land (x_1 \lor x_2 \lor x_5) \land (x_1 \lor x_3 \lor x_5) \land (x_3 \lor x_4 \lor x_5)$.

\begin{figure}[!h]
\centering
\begin{tikzpicture}

 \node[shape=circle,draw=black, thin, minimum size = 0.2cm, inner sep=0pt] (2) at (1,0) {};
 \node[shape=circle,draw=black, fill=black, thin, minimum size = 0.2cm, inner sep=0pt] (3) at (2,0) {};

 \node[shape=circle,draw=black, fill=black, thin, minimum size = 0.2cm, inner sep=0pt] (2a) at (4,0) {};
 \node[shape=circle,draw=black, thin, minimum size = 0.2cm, inner sep=0pt] (3a) at (5,0) {};

 \node[shape=circle,draw=black, fill=black, thin, minimum size = 0.2cm, inner sep=0pt] (2b) at (7,0) {};
 \node[shape=circle,draw=black, thin, minimum size = 0.2cm, inner sep=0pt] (3b) at (8,0) {};

 \node[shape=circle,draw=black, thin, minimum size = 0.2cm, inner sep=0pt] (2c) at (10,0) {};
 \node[shape=circle,draw=black, fill=black, thin, minimum size = 0.2cm, inner sep=0pt] (3c) at (11,0) {};

 \node[shape=circle,draw=black, fill=black, thin, minimum size = 0.2cm, inner sep=0pt] (2d) at (13,0) {};
 \node[shape=circle,draw=black, thin, minimum size = 0.2cm, inner sep=0pt] (3d) at (14,0) {};

   \path [-](2) edge node[left] {} (3); 
  
   \path [-](2a) edge node[left] {} (3a); 
  
   \path [-](2b) edge node[left] {} (3b);
 
   \path [-](2c) edge node[left] {} (3c);
 
   \path [-](2d) edge node[left] {} (3d);
   
   \node [black, below] at (2, -0.1) {$x_1$};
   \node [black, below] at (5, -0.1) {$x_2$};
   \node [black, below] at (8, -0.1) {$x_3$};
   \node [black, below] at (11, -0.1) {$x_4$};
   \node [black, below] at (14, -0.1) {$x_5$};

  \node[shape=circle,draw=black, thin, minimum size = 0.2cm, inner sep=0pt] (c1) at (3,1) {};
  \node[shape=circle,draw=black, thin, minimum size = 0.2cm, inner sep=0pt] (c2) at (6.5,-2) {};
  \node[shape=circle,draw=black, thin, minimum size = 0.2cm, inner sep=0pt] (c3) at (10,2) {};
  \node[shape=circle,draw=black, thin, minimum size = 0.2cm, inner sep=0pt] (c4) at (10,-1) {};

  \node [black, above] at (3, 1.1) {$c_1$};
  \node [black, below] at (6.5, -2.1) {$c_2$};
  \node [black, above] at (10, 2.1) {$c_3$};
  \node [black, below] at (10, -1.1) {$c_4$};

   \draw [-](c1) to [] (3);
   \draw [-](c1) to [] (3a);
   \draw [-](c1)  to [] (3b);
   \draw [-](c2)  to [] (3);
   \draw [-](c2)  to []  (3a);
   \draw [-](c2)  to [out=0, in=320]  (3d);
   \draw [-](c3)  to [out=180, in=120] (3);
   \draw [-](c3)  to [] (3d);
   \draw [-](c3)  to [] (3b);
   \draw [-](c4)  to []  (3b);
   \draw [-](c4) to [] (3c);
   \draw [-](c4)  to []  (3d);

    \node [black, below] at (1, -0.1) {$v_1$};
   \node [black, below] at (4, -0.1) {$v_2$};
   \node [black, below] at (7, -0.1) {$v_3$};
   \node [black, below] at (10, -0.1) {$v_4$};
   \node [black, below] at (13, -0.1) {$v_5$};
   
\end{tikzpicture}
\caption{The graph $G_{\phi}$ with four clauses 
$(x_1 \lor x_2 \lor x_3) \land (x_1 \lor x_2 \lor x_5) \land (x_1 \lor x_3 \lor x_5) \land (x_3 \lor x_4 \lor x_5)$. $\phi$ is satisfiable by setting $x_1, x_4$ to be true. The black vertices will receive a color from their respective lists in any valid CFCN$^*$-coloring of the graph.}
\label{fig 1 cn*}
\end{figure}
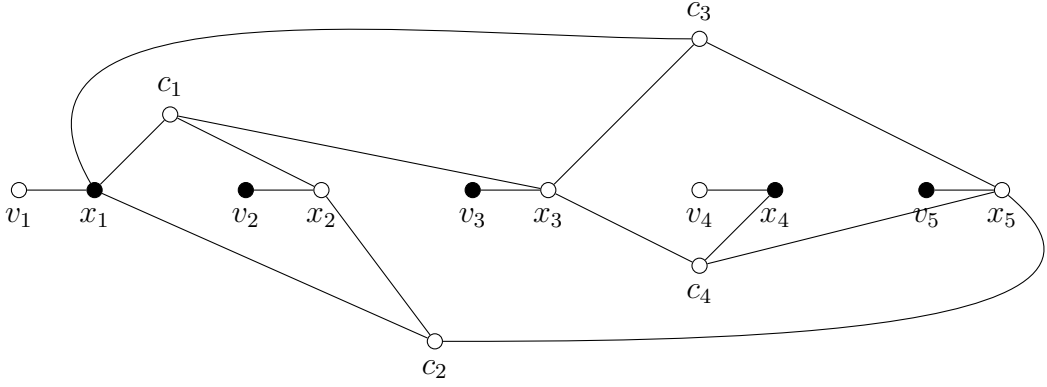

\begin{claim}
\label{claim 1 cfcn* chhosability planar bipar}
The formula  $\phi$ is a yes instance of the \textsc{Positive Planar 1-in-3-SAT Problem} if and only if $G''_{\phi}$ is 1-CFCN$^*$-choosable.   
\end{claim}
\noindent \textbf{Proof of claim.}
Suppose $\phi$ is a yes instance, that is, there is an assignment of truth values to its variables where each clause sees exactly one \textsc{True} variable and two \textsc{False} variables. Let $\mathcal{L}= \{L_v ~:~ v \in V(G''_{\phi})\}$ be a $1$-assignment for $G''_{\phi}$. 
Let $i \in \{1, \ldots n\}$. 
If $x_i=\textsc{True}$, then $\mathcal{L}$-color the vertex $x_i$. Otherwise, $\mathcal{L}$-color the vertex $v_i$. The reader may verify that, under this coloring, every vertex has exactly one vertex in its closed neighborhood that is colored.  

Suppose, $G''_{\phi}$ is 1-CFCN$^*$-choosable. Let $\mathcal{L}= \{L_v ~:~ v \in V(G''_{\phi})\}$ be a $1$-assignment for $G''_{\phi}$, where $L_v = \{c\}$, for every $v \in  V(G''_{\phi})$.
Consider an $\mathcal{L}$-CFCN$^*$ coloring $f$ of $G''_{\phi}$. 
We observe that, under the coloring $f$, the variable gadget for each variable $x_i$ has exactly one vertex colored; either the vertex $x_i$ or the vertex $v_i$. We set $x_i = \textsc{True}$ if the vertex $x_i$ is colored under $f$. Otherwise, we set $x_i = \textsc{False}$. Since $f$ is a valid CFCN$^*$ coloring of $G''_{\phi}$, every clause vertex has exactly one of its neighbors colored. Thus, under the above truth assignment, every clause will see exactly one \textsc{True} variable. 
\end{proof}

Now we show that the \textsc{2-CFCN$^*$-choosability problem} is NP-hard by reducing from the \textsc{1-CFCN$^*$-choosability problem}.
Below, we present a construction of a graph $H_G$ derived from a given graph $G$.
\\ \\
\noindent \textbf{Construction:} We construct $H_G$ by taking $12$ vertex-disjoint copies of $G$, denoted $G^a_{12}, G^b_{12}, G^a_{13}, G^b_{13}, G^a_{14}, G^b_{14}, G^a_{23}, G^b_{23}, G^a_{24}, G^b_{24}, G^a_{34}, G^b_{34}$, and adding four new vertices $v_1$, $v_2$, $v_3$, $v_4$. 
For every $1 \le i < j \le 4$, $z\in \{a,b\}$, we add an edge from (i) $v_i$ to every vertex in $G_{ij}^z$ and (ii) $v_j$ to every vertex in $G_{ij}^z$.
These are the only edges of $H_G$. Thus, $|E(H_G)| = 12 \cdot |E(G)| + |\{v_1, v_2, v_3, v_4\}| \cdot 6 \cdot |V(G)| = 12 \cdot |E(G)| + 24 \cdot |V(G)|$. 
Figure~\ref{fig: 2cn*} illustrates the construction of $H_G$ from $G$.

\begin{figure}[!h]
\centering
\begin{tikzpicture}

 \node[shape=circle,draw=black, thin, minimum size = 0.2cm, inner sep=0pt] (1) at (1,0) {};
 \node[shape=circle,draw=black,  thin, minimum size = 0.2cm, inner sep=0pt] (2) at (4,0) {};
 \node[shape=circle,draw=black,  thin, minimum size = 0.2cm, inner sep=0pt] (3) at (7,0) {};
 \node[shape=circle,draw=black, thin, minimum size = 0.2cm, inner sep=0pt] (4) at (10,0) {};

   \node [black, below] at (1, -0.1) {$v_1$};
   \node [black, below] at (4, -0.1) {$v_2$};
   \node [black, below] at (7, -0.1) {$v_3$};
   \node [black, below] at (10, -0.1) {$v_4$};
   
  \node[shape=rectangle,draw=black, thin, minimum size = 0.6cm, inner sep=0pt] (12a) at (0,3.2) {};
  \node[shape=rectangle,draw=black, thin, minimum size = 0.6cm, inner sep=0pt] (12b) at (1,3.2) {};
  \node[shape=rectangle,draw=black, thin, minimum size = 0.6cm, inner sep=0pt] (13a) at (2,3.2) {};
  \node[shape=rectangle,draw=black, thin, minimum size = 0.6cm, inner sep=0pt] (13b) at (3,3.2) {};
  \node[shape=rectangle,draw=black, thin, minimum size = 0.6cm, inner sep=0pt] (14a) at (4,3.2) {};
  \node[shape=rectangle,draw=black, thin, minimum size = 0.6cm, inner sep=0pt] (14b) at (5,3.2) {};
  \node[shape=rectangle,draw=black, thin, minimum size = 0.6cm, inner sep=0pt] (23a) at (6,3.2) {};
  \node[shape=rectangle,draw=black, thin, minimum size = 0.6cm, inner sep=0pt] (23b) at (7,3.2) {};
  \node[shape=rectangle,draw=black, thin, minimum size = 0.6cm, inner sep=0pt] (24a) at (8,3.2) {};
  \node[shape=rectangle,draw=black, thin, minimum size = 0.6cm, inner sep=0pt] (24b) at (9,3.2) {};
  \node[shape=rectangle,draw=black, thin, minimum size = 0.6cm, inner sep=0pt] (34a) at (10,3.2) {};
  \node[shape=rectangle,draw=black, thin, minimum size = 0.6cm, inner sep=0pt] (34b) at (11,3.2) {};
  
 \node [black, above] at (0, 3.5) {$G_{12}^a$}; 
 \node [black, above] at (1, 3.5) {$G_{12}^b$}; 
 \node [black, above] at (2, 3.5) {$G_{13}^a$}; 
 \node [black, above] at (3, 3.5) {$G_{13}^b$};
 \node [black, above] at (4, 3.5) {$G_{14}^a$};
 \node [black, above] at (5, 3.5) {$G_{14}^b$};
 \node [black, above] at (6, 3.5) {$G_{23}^a$};
 \node [black, above] at (7, 3.5) {$G_{23}^b$};
 \node [black, above] at (8, 3.5) {$G_{24}^a$};
 \node [black, above] at (9, 3.5) {$G_{24}^b$};
 \node [black, above] at (10, 3.5) {$G_{34}^a$};
 \node [black, above] at (11, 3.5) {$G_{34}^b$}; 

\path [very thick, dotted](1) edge node[left] {} (12a); 
\path [very thick, dotted](1) edge node[left] {} (12b); 
\path [very thick, dotted](1) edge node[left] {} (13a); 
\path [very thick, dotted](1) edge node[left] {} (13b); 
\path [very thick, dotted](1) edge node[left] {} (14a); 
\path [very thick, dotted](1) edge node[left] {} (14b); 

\path [very thick, dotted](2) edge node[left] {} (12a);
\path [very thick, dotted](2) edge node[left] {} (12b);
\path [very thick, dotted](2) edge node[left] {} (23a);
\path [very thick, dotted](2) edge node[left] {} (23b);
\path [very thick, dotted](2) edge node[left] {} (24a);
\path [very thick, dotted](2) edge node[left] {} (24b);

\path [very thick, dotted](3) edge node[left] {} (13a);
\path [very thick, dotted](3) edge node[left] {} (13b);
\path [very thick, dotted](3) edge node[left] {} (23a);
\path [very thick, dotted](3) edge node[left] {} (23b);
\path [very thick, dotted](3) edge node[left] {} (34a);
\path [very thick, dotted](3) edge node[left] {} (34b);

\path [very thick, dotted](4) edge node[left] {} (14a);
\path [very thick, dotted](4) edge node[left] {} (14b);
\path [very thick, dotted](4) edge node[left] {} (24a);
\path [very thick, dotted](4) edge node[left] {} (24b);
\path [very thick, dotted](4) edge node[left] {} (34a);
\path [very thick, dotted](4) edge node[left] {} (34b);

\end{tikzpicture}
\caption{The graph $H_G$. For $\ell \in [4]$, $1 \le i < j \le 4$, and $z \in \{a,b\}$, a dotted edge from $v_\ell$ to $G_{ij}^z$ indicates that $v_\ell$ is adjacent to every vertex in $G_{ij}^z$.}
\label{fig: 2cn*}
\end{figure}
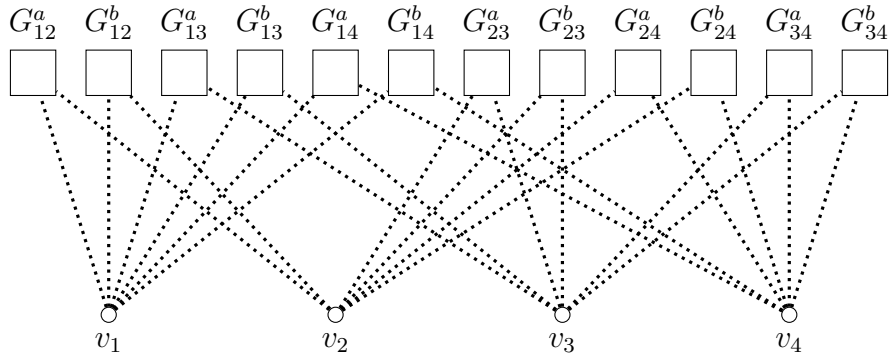

\begin{lemma}
\label{lem 2 cfcn* choosability} 
A graph $G$ is $1$-CFCN$^*$-choosable if and only if the graph $H_G$ is $2$-CFCN$^*$-choosable.
\end{lemma}

\begin{proof}
Suppose $G$ is 1-CFCN$^*$-choosable. We will show that $H_G$ is 2-CFCN$^*$-choosable. Let $\mathcal{L}=\{L_v~:~ v \in V(H_G)\}$ be a 2-assignment for $H_G$. We first color the vertices $v_1, v_2$, and $v_3$ by assigning each an arbitrary color from its list, ensuring that at most two vertices receive the same color. The vertex $v_4$ remains uncolored. The vertices $v_1, v_2$, and $v_3$ see their own color as a unique color. Let $c$ be the color given to $v_1$.
For a vertex $v_4$ to see a unique color in its closed neighborhood, we color exactly one vertex $w$ in $G_{14}^a$ with a color other than color $c$ from its list. The color of the vertex $w$ serves as a unique color in the closed neighborhood of $v_4$.  We now have two cases:\\
(i) If all three vertices $v_1, v_2, v_3$ receive distinct colors, then it can be verified that every vertex in $H_G$ has a unique color in its closed neighborhood. 
We do not color the remaining vertices in $H_G$ (only four vertices in $H_G$ are colored i.e., $v_1, v_2, v_3$, and $w$). 
\\
(ii) If instead two of the vertices among $v_1, v_2, v_3$ receive the same color, then we proceed as follows. Without loss of generality, assume that $v_1$ and $v_2$ receive color $c$. We then remove the color $c$ from the list of every vertex in $G_{12}^a$ and $G_{12}^b$. Note that some vertices in $G_{12}^a$ and $G_{12}^b$ may be left with a singleton list after this removal. Since $G_{12}^a$ and $G_{12}^b$ are $1$-CFCN$^*$-choosable, $G_{12}^a$ and $G_{12}^b$ admit $1$-CFCN$^*$-coloring under the reduced lists. 
We do not color the remaining vertices in $H_G$ ($v_1, v_2, v_3, w$ and all the vertices in $G_{12}^a$ and $G_{12}^b$ are the only colored vertices in $H_G$).
This yields a valid $\mathcal{L}$-CFCN$^*$-coloring of $H_G$, where every vertex in $G_{12}^a$ and $G_{12}^b$ finds a unique color within $G_{12}^a$ and $G_{12}^b$, respectively, and every vertex in the remaining ten copies of $G$ (namely, $G^a_{13}, G^b_{13}, \dots, G^b_{34})$ has a uniquely colored neighbor in the set $\{v_1, v_2, v_3\}$. 

Suppose $G$ is not 1-CFCN$^*$-choosable. We will show that $H_G$ is not 2-CFCN$^*$-choosable. Let $\mathcal{L}=\{L_v~:~ v \in V(H)\}$ be a 2-assignment for $H_G$, where $L_v= \{R, B\}$, for every $v\in V(H_G)$. We have the following two cases:
\\
(i) Consider an $\mathcal{L}$-coloring of $H_G$ that colors at least three vertices among $v_1, v_2, v_3, v_4$. WLOG, assume $v_1, v_2$, and $v_3$ are colored. Since $L_v=\{R, B\}$, $\forall v \in V(H_G)$, at least two of the three vertices $v_1, v_2, v_3$ receive the same color. 
Without loss of generality, assume that $v_1$ and $v_2$ receive color $R$, under the $\mathcal{L}$-coloring of $H_G$. Then for $G_{12}^a$ and $G_{12}^b$ to see a unique color, color $R$ should not be assigned to any vertex in $G_{12}^a$ or $G_{12}^b$. 
After the removal of color $R$ from every vertex's list in $G_{12}^a$ and $G_{12}^b$, it is left with only color $B$ in its list. 
Since neither $G_{12}^a$ nor $G_{12}^b$ is $1$-CFCN$^*$-choosable, by Proposition \ref{prop: 1 coloring choosability same} (ii), neither $G_{12}^a$ nor $G_{12}^b$ is $1$-CFCN$^*$-colorable. Therefore $G_{12}^a$ and $G_{12}^b$ are not list CFCN$^*$-colorable under the reduced list. Thus this case does not yield a valid $2$-CFCN$^*$-coloring of $H_G$. 
\\ 
(ii) Consider an $\mathcal{L}$-coloring of $H_G$ where at least two of the vertices among $v_1, v_2, v_3, v_4$ are uncolored. Without loss of generality, assume that $v_1$ and $v_2$ are uncolored under the $\mathcal{L}$-coloring of $H_G$. Then some of the vertices in $G_{12}^a$ and $G_{12}^b$ must be colored so that every vertex in $G_{12}^a$ and $G_{12}^b$ sees a unique color in its closed neighborhood. 
For the vertex $v_1$ (or $v_2$) to see a unique color in its closed neighborhood one color should be present exactly once among the vertices in $G_{12}^a$ and $G_{12}^b$. Assume the color $R$ is given to exactly one vertex in $G_{12}^a$. Then the color $R$ cannot be given to any vertex in $G_{12}^b$. So after the removal of color $R$ from every vertex list in $G_{12}^b$, every vertex in $G_{12}^b$ is left with only one color, i. e. the color $B$.
Since $G_{12}^b$ is not $1$-CFCN$^*$-choosable, by Proposition \ref{prop: 1 coloring choosability same} (ii), $G_{12}^b$ is not $1$-CFCN$^*$-colorable. Therefore, $G_{12}^b$ is not list CFCN$^*$-colorable under the reduced list. Thus this case does not yield a valid $2$-CFCN$^*$-coloring of $H_G$. 
\end{proof}

Combining Lemma \ref{lem 2 cfcn* choosability} and Theorem \ref{thm 1 cfcn* chhosability planar bipar}, we have the following result.

\begin{theorem}
\label{thm 2 cfcn* chhosability}
    The \textsc{$2$-CFCN$^*$-choosability problem} is NP-hard.
\end{theorem}

Now we show that the \textsc{2-CFON$^*$-choosability} problem is NP-hard on bipartite graphs by reducing from the \textsc{2-CFCN$^*$-choosability} problem.

\begin{definition}
    Given a graph $G$ with $V(G)= \{v_1, v_2, \dots, v_n\}$, the \emph{extended double cover} of $G$, denoted $D_G$, is the bipartite graph with bipartition $\{X, Y\}$, where $X=\{x_1, x_2, \dots, x_n\}$, $Y= \{y_1, y_2, \dots, y_n \}$ and $x_ix_j \in E(D_G)$ if and only if $j=i$ or $v_iv_j \in E(G)$.
\end{definition}

We now prove the following lemma.
\begin{lemma}
\label{lem 2 cfon* choosability} 
\textnormal{
(i)  A graph $G$ is $k$-CFCN$^*$-choosable if and only if its extended double cover $D_G$ is $k$-CFON$^*$-choosable, where $k \ge 1$. \\
(ii) A graph $G$ is $k$-CFCN-choosable if and only if its extended double cover $D_G$ is $k$-CFON-choosable, where $k \ge 1$.}
\end{lemma}

\begin{proof}
We prove (i). The proof of (ii) is analogous. 
Let $V(G)= \{v_1, v_2, \dots, v_n\}$. Let $D_G$ be the extended double cover of $G$ as defined above.

Suppose $G$ is $k$-CFCN$^*$-choosable. We will show that $D_G$ is $k$-CFON$^*$-choosable. Let $\mathcal{L}=\{L_v~:~ v \in V(D_G)\}$ be a $k$-assignment for $D_G$. 
Let $\mathcal{L}^X = \{L_{v_i}^X ~:~ v_i \in V(G)\}$, where $L_{v_i}^X = L_{x_i}$, be a $k$-assignment for $G$. 
Let $f^X$ be any valid $\mathcal{L}^X$-CFCN$^*$ coloring of $G$.
In a similar way, let $\mathcal{L}^Y = \{L_{v_i}^Y ~:~ v_i \in V(G)\}$, where $L_{v_i}^Y = L_{y_i}$, be a $k$-assignment for $G$. Let $f^Y$ be any valid $\mathcal{L}^Y$-CFCN$^*$ coloring of $G$.
We now obtain an $\mathcal{L}$-CFON$^*$-coloring $f$ of $D_G$ as described below: $f(x_i)=f^X(v_i)$ and $f(y_i)=f^Y(v_i)$, $\forall i \in [n]$. 
We claim that $f$ is a CFON$^*$-coloring of $D_G$. Consider a vertex $x_i \in X$. Suppose $v_j$ is a uniquely colored vertex in the closed  neighborhood of $v_i$ in $G$, under the coloring $f^Y$. Then, under $f$, $y_j$ is a uniquely colored neighbor in the open neighborhood of $x_i$. 
In a similar way, one can show that every vertex in $Y$ sees a unique color in its open neighborhood under $f$.

Suppose $D_G$ is $k$-CFON$^*$-choosable.
We will show that $G$ is $k$-CFCN$^*$-choosable. Let $\mathcal{L}=\{L_v~:~ v \in V(G)\}$ be a $k$-assignment for $G$.
Let $\mathcal{L}'=\{L'_v~:~ v \in V(D_G)\}$, where $L'_{x_i} = L'_{y_i} = L_{v_i}$, $\forall i \in [n]$, be a $k$-assignment for $D_G$.
Let $f'$ be an $\mathcal{L}'$-CFON$^*$-coloring of $D_G$. We now obtain an $\mathcal{L}$-CFCN$^*$-coloring $f$ of $G$ as described below:
$f(v_i)=f'(x_i)$, $\forall i \in [n]$.  Consider a vertex $v_i \in V(G)$. Let $x_j$ be a uniquely colored vertex in the open  neighborhood of $y_i$ under $f'$. Then $v_j$ is a uniquely colored neighbor in the closed neighborhood of $v_i$ in $G$, under $f$. 
\end{proof}

Combining Lemma \ref{lem 2 cfon* choosability} and Theorem \ref{thm 2 cfcn* chhosability}, we have the following result.

\begin{theorem}
\label{thm 2 cfon* choosability}
The \textsc{$2$-CFON$^*$-choosability problem} is NP-hard on bipartite graphs.
\end{theorem}

\bibliographystyle{plain}
\bibliography{Thesis_reference}
\end{document}